\documentclass[11pt]{amsart}
\UseRawInputEncoding
\usepackage{graphicx}
\usepackage{mathrsfs}
\usepackage{hyperref}
\newtheorem{thm}{Theorem}[section]

\theoremstyle{definition}
\newtheorem{defn}[thm]{Definition}
\theoremstyle{remark}
\newtheorem{rem}[thm]{Remark}

\begin{document}

\title[Convergence of positive series]{Necessary and sufficient conditions for the convergence of positive series}%
\author{Vyacheslav M. Abramov}%
\address{24 Sagan Drive, Cranbourne North, Victoria, 3977, Australia}%

\email{vabramov126@gmail.com}%

\begin{abstract}
We provide new necessary and sufficient conditions for the convergence of positive series developing Bertran--De Morgan and Cauchy type tests given in [M. Martin, \emph{Bull. Amer. Math. Soc.} 47(1941), 452–-457] and [L. Bourchtein et al, \emph{Int. J. Math. Anal.} 6(2012), 1847--1869]. The obtained result enables us to extend the known conditions for recurrence and transience of birth-and-death processes given in [V. M. Abramov, \emph{Amer. Math. Monthly} 127(2020) 444--448].
\end{abstract}
\maketitle

\section{Introduction}
Let
\begin{equation}\label{4}
\sum_{n=1}^{\infty}a_n
\end{equation}
be a series with positive terms, for which $a_{n+1}\leq a_n$, $n\geq1$.

The ratio tests of convergence or divergence of \eqref{4} are widely known and go back to the works of d'Alembert and Cauchy as well as many other researchers in the eighteenth and nineteenth centuries such as Raabe, Gauss, Bertrand, De Morgan and Kummer. They are classified into the De Morgan hierarchy \cite{B, BM}. The extended Bertrand--De Morgan test is the last test in this hierarchy. It was originally established in \cite{M}. An elementary proof of this test, its connection with Kummer's test, as well as its application to birth-and-death processes is given in \cite{A}. Further generalization of the extended Bertrand--De Morgan test based on the connection with the class of regularly varying functions is given in \cite{ACO}. Similarly defined hierarchy of Cauchy's tests is obtained in \cite{BBNV}.

In the present note, we establish necessary and sufficient conditions for convergence of positive series that generalize the original version of the extended Bertrand--De Morgan test \cite{A,M} and Bertrand--De Morgan--Cauchy test \cite{BBNV}. The first theorem on a necessary and sufficient condition for convergence of series was obtained by Cauchy \cite{C}, widely known as Cauchy's convergence test.
Later, at the beginning of the twentieth century, necessary and sufficient conditions for convergence of positive series were obtained by Brink \cite{Br1, Br2}. The statements of the aforementioned theorems \cite{Br1, Br2} involve the convergence of double or triple improper integrals having the complex expressions. Furthermore, the test in \cite[page 47]{Br2} is based on the double ratios $r_n=a_{n+1}/a_n$ and $R_n=r_{n+1}/r_n$. This made the areas of their applications very limited by problems having technical nature. The basic theorem of Brink \cite{Br1} was developed in \cite{R}. Another theorem of Brink \cite{Br2} is mentioned in \cite{M} as starting point for the derivation of the main result. A theorem of Tong \cite{T} that develops Kummer's test also provides necessary and sufficient conditions for convergence or divergence of positive series. However, its practical applications to real world problems is hard, since it requires a special construction that involves an auxiliary sequence.

The idea of our approach is to use the extended Bertrand--De Morgan--Cauchy test in the form of an explicit inequality with the following derivation of the expression for $a_n$ for large $n$. Then the test can be naturally adapted to the problems from applied areas.
Specifically, the result obtained in this note improves the conditions of recurrence or transience for birth-and-death processes given in \cite[Theorem 3]{A}, thus extending the class of birth-and-death processes for which that condition can be established.

Below we recall the extended Bertrand--De Morgan test in the formulation given in \cite{A}.
Let $\ln_{(k)}z$ denote the $k$th iterate of natural logarithm, i.e. $\ln_{(1)}z$ $=\ln z$, and $\ln_{(k)}z=\ln(\ln_{(k-1)}z)$, $k\geq2$.

\begin{thm}\label{t0}
Suppose that for all large $n$ and some $K\geq1$
\begin{equation}\label{0}
\frac{a_n}{a_{n+1}}=1+\frac{1}{n}+\frac{1}{n}\sum_{i=1}^{K-1}\frac{1}{\prod_{k=1}^{i}\ln_{(k)}n}+\frac{s_n}{n\prod_{k=1}^{K}\ln_{(k)}n}.
\end{equation}
(The empty sum is set to $0$.) Then \eqref{4} converges if $\liminf_{n\to\infty} s_n > 1$, and it diverges if $\limsup_{n\to\infty} s_n < 1$.
\end{thm}

 The statement of Theorem \ref{t0} is a ratio test for the fraction $a_n/a_{n+1}$ when $n$ is large. For the purpose of the present paper
 an application of this theorem is insufficient (see explanation given later in Remark \ref{r1}), and
 we need a stronger version of the theorem, in which the fraction on the left-hand side of \eqref{0} is replaced with $\sqrt[-n]{a_n}$, i.e. the ratio test is replaced with a variant of Cauchy's root test. This stronger version presents Bertrand--De Morgan--Cauchy test formulated below.
\begin{thm}\label{t01}
Suppose that for all large $n$ and some $K\geq1$
\begin{equation*}\label{00}
\sqrt[-n]{a_n}=1+\frac{1}{n}+\frac{1}{n}\sum_{i=1}^{K-1}\frac{1}{\prod_{k=1}^{i}\ln_{(k)}n}+\frac{s_n}{n\prod_{k=1}^{K}\ln_{(k)}n}.
\end{equation*}
Then \eqref{4} converges if $\liminf_{n\to\infty} s_n > 1$, and it diverges if $\limsup_{n\to\infty} s_n < 1$.
\end{thm}

A lengthy proof of this theorem in another formulation can be found in \cite{BBNV}. For the purpose of the present paper we provide a simple alternative proof given below.
\smallskip

 Theorem \ref{t01} follows, since $\sqrt[-n]{a_n}=\mathrm{e}^{-\frac{1}{n}\ln a_n}$ leads to the equation
\begin{equation}\label{16}
\mathrm{e}^{-\frac{1}{n}\ln a_n}=1+\frac{1}{n}+\frac{1}{n}\sum_{i=1}^{K-1}\frac{1}{\prod_{k=1}^{i}\ln_{(k)}n}+\frac{s_n}{n\prod_{k=1}^{K}\ln_{(k)}n}, \quad K\geq1.
\end{equation}
With some algebra (see Appendix A) for large $n$ we arrive at
\begin{equation}\label{15}
a_n=\begin{cases}\frac{C_n}{\left(n\prod_{k=1}^{K-2}\ln_{(k)}n\right)\ln_{(K-1)}^{s_n}n}, &K\geq2,\\
\frac{C_n}{n^{s_n}}, &K=1,
\end{cases}
\end{equation}
where the empty product is set to 1, and $C_n$ are the constants satisfying the property $\lim_{n\to\infty}C_n=\mathrm{e}^{-1}$. Then, convergence or divergence of \eqref{4} follows from the integral test for convergence or divergence of series applied to \eqref{15}.
\smallskip

Theorem \ref{t01} is then used to prove the main result of this paper formulated in the next section.

The cases $\liminf_{n\to\infty} s_n=1$ or $\limsup_{n\to\infty} s_n=1$ remain undefined for both Theorems \ref{t0} and \ref{t01}. The main result of this paper covers all possible limit cases including these undefined ones.

\smallskip
The rest of the note is structured into two sections. In Section \ref{S2}, the main result of this note is proved. In Section \ref{S3}, an application of the main result to birth-and-death processes is discussed.

\section{Necessary and sufficient conditions for convergence of \eqref{4}}\label{S2}
The theorem given below provides necessary and sufficient conditions for the convergence of positive series.

\smallskip
Let $\mathscr{N}\subset\mathbb{N}$, and let $N(n)$ denote the number of integers in $\mathscr{N}$ not greater than $n$.

\begin{defn}\label{defn1}
We say that the set $\mathscr{N}$ contains \textit{almost all} elements of $\mathbb{N}$, if $\lim_{n\to\infty}N(n)/n=1$.
\end{defn}

\begin{defn}
We say that the set $\mathscr{N}$ contains \textit{strongly almost all} elements of $\mathbb{N}$, if $N(n)=n+O(1)$ as $n\to\infty$.
\end{defn}

Elementary examples of these definitions are $N(n)=n-\lfloor \ln n\rfloor$, where $\lfloor a\rfloor$ denotes the integer part of $a$, and $N(n)=n-c_n$, where $c_n<n$ is a bounded sequence of integers. In the first case, $\mathscr{N}$ contains almost all elements. In the second one, $\mathscr{N}$ contains strongly almost all elements.

\begin{thm}\label{t1}
Suppose that  there exist constants $r$ and $\alpha>0$ such that for all values $n$ we have $a_{n}<rn^{-\alpha}$.
Then \eqref{4} converges if there exist integer $K\geq1$ and real $c>1$ such that for strongly almost all $n$
\begin{equation}\label{13}
\sqrt[-n]{a_n}\geq1+\frac{1}{n}+\frac{1}{n}\sum_{i=1}^{K-1}\frac{1}{\prod_{k=1}^{i}\ln_{(k)}n}+\frac{c}{n\prod_{k=1}^{K}\ln_{(k)}n},
\end{equation}
and only if \eqref{13}
is satisfied for almost all $n$.
\end{thm}

\begin{proof}
Assume that $N(n)/n=1+O(1/n)$,  and $\mathbb{N}\setminus\mathscr{N}$ is the subset of indices for which \eqref{13} is not satisfied.
Write
\begin{equation}\label{8}
\sum_{n=1}^{\infty}a_n=\underbrace{\sum_{n\in\mathscr{N}}a_n}_{=I_1}+\underbrace{\sum_{n\in\mathbb{N}\setminus\mathscr{N}}a_n}_{=I_2}.
\end{equation}
Since $N(n)/n=1+O(1/n)$, then the fraction of the terms satisfying the inequality $a_{n}<rn^{-\alpha}$ and not satisfying \eqref{13} is $O(1/n)$ as $n\to\infty$, and hence
$
I_2<R\sum_{n\in\mathbb{N}}n^{-1-\alpha}<\infty
$
for some constant $R$.
Then $I_1$ contains the only terms, for which \eqref{13} is satisfied. Combining these terms, we have the presentation
\begin{equation}\label{2}
I_1=\sum_{i_1=j_1}^{n_1}a_{i_1}+\sum_{i_2=n_1+j_2}^{n_2}a_{i_2}+\ldots,
\end{equation}
where the series of sums is given over the indices belonging to $\mathscr{N}$. Since $N(n)=n+O(1)$, then, as $m\to\infty$, the difference between the indices $n_{m}+j_{m+1}$ and $n_m$, that are the lower index of the $m+1$st sum and the upper index of the $m$th sum, respectively, must be bounded. That is, $j_m=O(1)$. Hence taking into account the expression on the right-hand side of \eqref{13} we obtain the estimate
\begin{equation}\label{3}
\sqrt[-n_m]a_{n_m}=\sqrt[-n_{m}-j_{m+1}]a_{n_m}+O\left(\frac{1}{n_m^2}\right).
\end{equation}
Specifically, the presence of the remainder term $O(n_m^{-2})$ in \eqref{3} is explained by the fact that
\[
\left[1+\frac{1}{n_m}+o\left(\frac{1}{n_m}\right)\right]-\left[1+\frac{1}{n_m+j_{m+1}}+o\left(\frac{1}{n_m+j_{m+1}}\right)\right]=O\left(\frac{1}{n_m^2}\right).
\]

Then, renumbering the terms in $I_1$ and taking into account estimate \eqref{3}, we arrive at the new series
$
\sum_{n=1}^\infty a_n^\prime
$
that approximates $I_1$.

According to \eqref{13} there exist $c>1$ and integers $K$ and $n_0$ such that for all $n>n_0$
\[
\sqrt[-n]{a_{n}^\prime}\geq1+\frac{1}{n}+\frac{1}{n}\sum_{i=1}^{K-1}\frac{1}{\prod_{k=1}^{i}\ln_{(k)}n}+\frac{c}{n\prod_{k=1}^{K}\ln_{(k)}n},
\]
and the sufficient condition follows by application of Theorem \ref{t01}.

For the necessary condition, we are to prove that if no such $K$ that \eqref{13} is satisfied with $c>1$ for almost all $n$, then series \eqref{4} diverges. Suppose that \eqref{13} is satisfied with $c>1$ and $K\geq1$ only for some $\mathscr{N}\subset\mathbb{N}$ such that $\lim_{n\to\infty}N(n)/n=\alpha<1$. Then,
$
\sum_{n=1}^{\infty}a_n=I_1+I_2> I_2.
$
Write
\[
I_2=\sum_{i_1=j_1}^{n_1}a_{i_1}+\sum_{i_2=n_1+j_2}^{n_2}a_{i_2}+\ldots,
\]
where the series of sums is given over the indices belonging to $\mathbb{N}\setminus\mathscr{N}$. Under the assumption $\lim_{n\to\infty}N(n)/n=\alpha<1$, the fraction of the terms in the series that are not satisfied \eqref{13} is proportional to $1-\alpha>0$. Let $N_m=n_m-n_{m-1}-j_m+1$ be the number of consecutive terms in the sum $\sum_{i_{m}=n_{m-1}+j_m}^{n_m}a_{i_m}$. We have the following cases:
\begin{enumerate}
\item [(a)]  the sequences $N_m$ and $j_m$ are bounded;
\item [(b)]  $\limsup_{m\to\infty}N_m=\infty$.
\end{enumerate}

To demonstrate these cases, we provide the following example. Consider a series, the terms of which for all large $n$ satisfy the relation
\begin{equation}\label{21}
\sqrt[-n]{a_n}=1+\frac{1}{n}+\frac{1}{n}\sum_{i=1}^{K-1}\frac{1}{\prod_{k=1}^{i}\ln_{(k)}n}+\frac{1+\epsilon_n}{n\prod_{k=1}^{K}\ln_{(k)}n}, \quad K\geq1,
\end{equation}
where $\epsilon_n$ is a vanishing sequence. Apparently, \eqref{21} can be rewritten in the form
\[
\sqrt[-n]{a_n}=1+\frac{1}{n}+\frac{1}{n}\sum_{i=1}^{K}\frac{1}{\prod_{k=1}^{i}\ln_{(k)}n}+\frac{\epsilon_n\ln_{(K+1)}n}{n\prod_{k=1}^{K+1}\ln_{(k)}n}, \quad K\geq1.
\]
Hence, according to Theorem \ref{t01}, \eqref{4} converges if $\epsilon_n>C\big(\ln_{(K+1)}n\big)^{-1}$, $C>1$, and it diverges if $\epsilon_n\leq\big(\ln_{(K+1)}n\big)^{-1}$. (The last statement does not follow directly from the formulation of Theorem \ref{t01}, but easily follows from the derivations provided in Appendix A.)

If $\epsilon_n=\big(\ln_{(K+1)}n\big)^{-1}+\epsilon_n^\prime$, where $\epsilon_n^\prime$ is another vanishing sequence, then the problem reduces to that considered above involving the additional term in presentation \eqref{21}. That is,
\[
\sqrt[-n]{a_n}=1+\frac{1}{n}+\frac{1}{n}\sum_{i=1}^{K}\frac{1}{\prod_{k=1}^{i}\ln_{(k)}n}+\frac{1+\epsilon_n^\prime}{n\prod_{k=1}^{K+1}\ln_{(k)}n}.
\]
Due to the recursion, there can be any number of additional terms.
\smallskip

Now, to illustrate case (a), one can imagine that $\epsilon_n$ for $n\geq n_0$, where $n_0$ is some odd large number, satisfies the following property: $\epsilon_n>C\big(\ln_{(K+1)}n\big)^{-1}$, $C>1$, if $n$ is even, and $\epsilon_n\leq\big(\ln_{(K+1)}n\big)^{-1}$ if $n$ is odd. Then the terms of the series, for which $\epsilon_n>C\big(\ln_{(K+1)}n\big)^{-1}$, $C>1$, belong to $I_1$, while the remaining terms, for which $\epsilon_n\leq\big(\ln_{(K+1)}n\big)^{-1}$, belong to $I_2$. That is,
\[
I_1=\sum_{k=(n_0+1)/2}^{\infty}a_{2k}, \quad I_2=\sum_{k=(n_0+1)/2}^{\infty}a_{2k-1}.
\]
The terms of $I_2$ can be further renumbered as $a_{n_0}^\prime$, $a_{n_0+1}^\prime$,\ldots such that the terms of a new series $\sum_{j=n_0}^{\infty}a_j^\prime$ satisfy the inequality
\begin{equation}\label{18}
\sqrt[-j]{a_j^\prime}\leq1+\frac{1}{j}+\frac{1}{j}\sum_{i=1}^{K}\frac{1}{\prod_{k=1}^{i}\ln_{(k)}j}.
\end{equation}
Apparently, presentation \eqref{18} is true in the general situation under the conditions described in case (a).

Case (b) also can be derived on the basis of presentation \eqref{21}. For instance, one can assume that starting from $n_0$ we have the following alternate relationships given in Table \ref{tab1}.
\begin{table}
    \begin{center}
            \caption{Correspondence between the values of $n$ and inequalities for $\epsilon_n$ ($n\geq n_0$, $C>1$)} \label{tab1}
        \begin{tabular}{l|l}\hline\hline
\text{Value of} \ $n$  & \text{Corresponding inequality for} \ $\epsilon_n$\\
\hline\hline
$n_0$             & $\epsilon_n>C(\ln_{(K+1)}n)^{-1}$\\
\hline
$n_0+1$           & $\epsilon_n\leq(\ln_{(K+1)}n)^{-1}$\\
\hline
$n_0+2$           & $\epsilon_n>C(\ln_{(K+1)}n)^{-1}$\\
$n_0+3$           & \\
\hline
$n_0+4$           & $\epsilon_n\leq(\ln_{(K+1)}n)^{-1}$\\
$n_0+5$           & \\
\hline
$n_0+6$           & $\epsilon_n>C(\ln_{(K+1)}n)^{-1}$\\
$n_0+7$           & \\
$n_0+8$           & \\
\hline
$n_0+9$           & $\epsilon_n\leq(\ln_{(K+1)}n)^{-1}$\\
$n_0+10$          & \\
$n_0+11$          & \\
\hline
\ldots            & \ldots\\
       \hline
        \end{tabular}
    \end{center}
\end{table}
Based on this table one can see that, as $m$ tends to infinity, there exists an increasing to infinity subsequence $N_{m_1}$, $N_{m_2}$,\ldots, resulting in the series of sums $\sum_{i=n_{m_1-1}+j_{m_1}}^{n_{m_1}}a_{i}$, $\sum_{i=n_{m_2-1}+j_{m_2}}^{n_{m_2}}a_{i}$,\ldots with increasing to infinity number of terms.

Notice that in both of these examples the assumption $a_{n+1}\leq a_n$ is supported. This is seen from the presentation for $a_n$ given by \eqref{15}.

\smallskip
Hence, following cases (a) and (b), without loss of generality it can be assumed that for the original series $
\sum_{n=1}^{\infty}a_n
$, there is no $K\geq1$ such that \eqref{13} is satisfied for all large $n$.

\smallskip

Thus, following the assumption that \eqref{13} is not satisfied, our study reduces to the following two cases:
\begin{enumerate}
\item [(i)] there exists $K\geq1$ and $n_0$ such that for all $n\geq n_0$
\[
\sqrt[-n]{a_n}\leq1+\frac{1}{n}+\frac{1}{n}\sum_{i=1}^{K-1}\frac{1}{\prod_{k=1}^{i}\ln_{(k)}n}+\frac{c^*}{n\prod_{k=1}^{K}\ln_{(k)}n},
\]
where $c^*\leq1$;

\item [(ii)] for large $n$ and $K$
\[
\begin{aligned}
\sqrt[-n]{a_n}=1+\frac{1}{n}&+\frac{1}{n}\sum_{i=1}^{K}\frac{1}{\prod_{k=1}^{i}\ln_{(k)}n},
\end{aligned}
\]
where $n$ and $K$ are given such that $\ln_{(K)}n$ is large.
\end{enumerate}

In case (i), if $c^*<1$, then the series diverges due to representation \eqref{15}. If $c^*=1$, then the inequality in (i) is rewritten
\[
\sqrt[-n]{a_n}\leq1+\frac{1}{n}+\frac{1}{n}\sum_{i=1}^{K}\frac{1}{\prod_{k=1}^{i}\ln_{(k)}n},
\]
and following the derivation similar to that given in Appendix A, we obtain the inequality
\[
a_{n}\geq C\frac{1}{n\prod_{k=1}^{K-1}\ln_{(k)}n}, \quad K\geq1,
\]
where $C>0$ is some constant. Hence \eqref{4} diverges.
\smallskip

In case (ii), we are to consider the sequence
\[
\begin{aligned}
\sqrt[-n]{a_n(K)}=1+\frac{1}{n}&+\frac{1}{n}\sum_{i=1}^{K}\frac{1}{\prod_{k=1}^{i}\ln_{(k)}n},
\end{aligned}
\]
assuming that both $K$ and $n$ increase to infinity, where $n$ and $K$ are chosen such that $\ln_{(K)}n$ is large. Let $n_0(K)$ be such a number that $\ln_{(K)}n_0(K)>0$. Then
\[
\sum_{n=n_0(K)}^{\infty}a_{n}(K)\geq C\sum_{n=n_0(K)}^{\infty}\frac{1}{n\prod_{k=1}^{K-1}\ln_{(k)}n}=\infty,
\]
and for any increasing sequence $K_1<K_2<\ldots$ we have
\[
\lim_{i\to\infty}\sum_{n=n_0(K_i)}^{\infty}a_{n}(K_i)\geq C\lim_{i\to\infty}\sum_{n=n_0(K_i)}^{\infty}\frac{1}{n\prod_{k=1}^{K_i-1}\ln_{(k)}n}=\infty,
\]
where $C$ is an absolute constant for all $i$. Thus if a series is convergent, then it must be presented by \eqref{13} with some $c>1$ and integer $K\geq1$ for almost all $n$.
\end{proof}

\begin{rem}\label{r1} For the `if' condition of Theorem \ref{t1}, the root function $\sqrt[-n]{a_n}$ on the left-hand side of \eqref{13} can be replaced with
the fraction $a_n/a_{n+1}$.
Indeed, assume that for all $n\in\mathscr{N}$ we have the inequality
\[
\frac{a_{n}}{a_{n+1}}\geq1+\frac{1}{n}+\frac{1}{n}\sum_{k=1}^{K-1}\frac{1}{\prod_{j=1}^{k}\ln_{(j)}n}+\frac{c}{n\prod_{k=1}^{K}\ln_{(k)}n},
\]
where $c>1$, $K\geq1$. Then, presentation \eqref{2} implies that $n_i+j_{i+1}\leq n\leq n_{i+1}$, $i\geq0$, $n_0=0$. For the boundary elements in the sums of $I_1$, we have the similar inequality
\[
\frac{a_{n_m}}{a_{n_m+j_{m+1}}}\geq1+\frac{1}{n_m}+\frac{1}{n_m}\sum_{k=1}^{K-1}\frac{1}{\prod_{j=1}^{k}\ln_{(j)}n_m}+\frac{c}{n_m\prod_{k=1}^{K}\ln_{(k)}n_m},
\]
which is true due to the convention $a_{n+1}\leq a_n$ for all $n\geq1$. Then, renumbering the terms in $I_1$ we arrive at the required result.

However for the `only if' condition, $\sqrt[-n]{a_n}$ cannot be replaced with $a_n/a_{n+1}$. It is seen from the following simple example. Suppose that for series \eqref{4} we have $a_{2k-1}=a_{2k}$, $k=1,2,\ldots$, and for all large $k$
\[
\frac{a_{2k}}{a_{2k+1}}\geq1+\frac{1}{2k}+\frac{1}{2k}\sum_{i=1}^{K-1}\frac{1}{\prod_{j=1}^{i}\ln_{(j)}(2k)}+\frac{c}{n\prod_{j=1}^{K}\ln_{(j)}(2k)}, \quad K\geq1, \quad c>1.
\]
Then, despite $\lim_{n\to\infty}N(n)/n=1/2$, the series is convergent.
\end{rem}

\section{Application}\label{S3}

Theorem \ref{t1} can be used to improve the conditions of recurrence and transience for Markov processes and, in particular, birth-and-death processes.

Consider a birth-and-death process, in which the birth and death parameters $\lambda_n$ and $\mu_n$ all are in $(0, \infty)$. It is known \cite[page 370]{K} that a birth-and-death process is transient if and only if
\[
\sum_{n=1}^{\infty}\prod_{k=1}^{n}\frac{\mu_k}{\lambda_k}<\infty.
\]
So, Theorem \ref{t1} can be immediately applied, if the required conditions are satisfied. Then for the `if' condition, it is convenient to use the fraction $a_n/a_{n+1}$ on the left-hand side of \eqref{13} rather than $\sqrt[-n]{a_n}$, since $a_n/a_{n+1}$ in this case reduces to $\lambda_{n+1}/\mu_{n+1}$ (see Remark \ref{r1}).

Note also that the condition $\prod_{k=1}^{n}({\mu_k}/{\lambda_k})<rn^{-\alpha}$ can be replaced by the following simpler conditions: (i) $\mu_n/\lambda_n$ converges to $1$ as $n\to\infty$, and (ii) there exist $\alpha>0$ and $n_0$ such that for all $n>n_0$ we have $\ln({\mu_n}/{\lambda_n})<-\alpha\ln n/{n}$.

\appendix
\section{Derivation of \eqref{15}}

From \eqref{16} we have
\begin{equation}\label{A.1}
\ln a_n=-n\ln\left(1+\frac{1}{n}+\frac{1}{n}\sum_{i=1}^{K-1}\frac{1}{\prod_{k=1}^{i}\ln_{(k)}n}+\frac{s_n}{n\prod_{k=1}^{K}\ln_{(k)}n}\right), \quad K\geq1.
\end{equation}
From Taylor's expansion of the logarithm on the right-hand side of \eqref{A.1} we obtain
\[
\ln a_n=-1-\sum_{i=1}^{K-1}\frac{1}{\prod_{k=1}^{i}\ln_{(k)}n}-\frac{s_n}{\prod_{k=1}^{K}\ln_{(k)}n}+O\left(\frac{1}{n}\right), \quad K\geq1.
\]
Hence,
\[
a_n=\begin{cases}\mathrm{e}^{-1+O(1/n)}\frac{1}{\left(n\prod_{k=1}^{K-2}\ln_{(k)}n\right)\ln_{(K-1)}^{s_n}n}, &K\geq2,\\
\mathrm{e}^{-1+O(1/n)}\frac{1}{n^{s_n}} &K=1.
\end{cases}
\]
The result follows.

\section*{Acknowledgement} The author thanks the reviewer for careful reading the paper and the comprehensive report.

\end{document}